\documentclass
{amsart}
\usepackage{amssymb, amsmath, amsfonts, amsthm, esint, mathtools}
\usepackage[usenames,dvipsnames,svgnames,table]{xcolor}
\usepackage[margin=1in]{geometry} 
\usepackage[utf8]{inputenc}
\usepackage{microtype}
\usepackage{enumerate}
\usepackage{multicol}
\usepackage{enumerate}
\usepackage{bbm}
\usepackage{tikz}
\usepackage{pgfplots}
\usepackage{nicefrac,microtype}
	\newcommand{\nfrac}{\nicefrac}
	\newcommand{\sfrac}{\nfrac}
\usepackage{cite}

\mathtoolsset{showonlyrefs}

\usepackage[colorlinks,linkcolor = blue,citecolor = ForestGreen]{hyperref}
\usepackage{zref-clever}
	\newcommand{\Cref}{\zcref[S]}

\usepackage{comment}

\newcommand{\abs}[1]{\left\vert#1\right\vert}
\newcommand{\norm}[1]{\left\|#1\right\|}  
\newcommand{\set}[1]{\left\{ #1 \right\}}
\newcommand{\brak}[1]{\left\langle #1 \right\rangle}
\newcommand{\vv}{\brak{v}}

\newcommand{\vvd}{\brak{\cdot}}

\newcommand{\be}{\begin{equation}}
\newcommand{\ee}{\end{equation}}

\newcommand{\dd}{\, {\rm d}}

\newcommand{\1}{\mathbbm{1}}

\newcommand{\eps}{\varepsilon}
\renewcommand{\epsilon}{\eps}

\newcommand{\R}{\ensuremath{{\mathbb R}}}

\renewcommand{\S}{\ensuremath{{\mathcal S}}}

\DeclareMathOperator{\Id}{Id}

\DeclareMathOperator{\Tr}{Tr}

\DeclareMathOperator{\QL}{Q_{\rm L}}

\newcounter{num} \numberwithin{num}{section}

\newtheorem{theorem}[num]{Theorem}
\newtheorem{heuristic}[num]{Heuristic Result}
\newtheorem{proposition}[num]{Proposition}
\newtheorem{lemma}[num]{Lemma}
\newtheorem{corollary}[num]{Corollary}
\theoremstyle{definition}

\theoremstyle{remark}
\newtheorem{remark}[num]{Remark}

\numberwithin{equation}{section}

\AddToHook{env/theorem/begin}
    {\zcsetup{countertype={num=theorem}}}
\AddToHook{env/proposition/begin}
    {\zcsetup{countertype={num=proposition}}}
\AddToHook{env/lemma/begin}
	{\zcsetup{countertype={num=lemma}}}
\AddToHook{env/definition/begin}
	{\zcsetup{countertype={num=definition}}}
\AddToHook{env/remark/begin}
	{\zcsetup{countertype={num=remark}}}
\AddToHook{env/corollary/begin}
	{\zcsetup{countertype={num=corollary}}}
\AddToHook{env/heuristic/begin}
	{\zcsetup{countertype={num=heuristic}}}
    \zcRefTypeSetup{heuristic}{
		Name-sg = Heuristic Result,
		name-sg = heuristic result,
		Name-pl = Heuristic Results,
		name-pl = heuristic results
	}

\newenvironment{NB}{%
    \par\vspace{\baselineskip}\noindent 
    \textbf{N.B.} \ignorespaces 
    \begin{itshape}
}{%
    \end{itshape}
    \par\vspace{\baselineskip}
}

\author{William Golding and Christopher Henderson}
\title[Hydrodynamic implosions and the Landau equation]{On Hydrodynamic Implosions and the Landau-Coulomb Equation}
\address[William Golding]
    {Department of Mathematics, The University of Chicago,
    Chicago, IL 60615, USA}
    \email{wgolding@uchicago.edu}
\address[Christopher Henderson]
    {Department of Mathematics, The University of Maryland, 
    College Park, MD 20742, USA}
    \email{ckhend@umd.edu}
\thanks{\textbf{Acknowledgments:} The authors would like to thank Luis Silvestre for comments on an early draft and his invaluable insight about the applicability of the continuation criterion in conjunction with the hydrodynamic limit. CH would like to thank Dave Levermore for many enlightening discussions, as well as his excellent slides on the history of kinetic equations and fluid dynamics.  CH was supported by NSF grants DMS-2337666 and DMS-2204615.}

\usepackage{natbib}

\begin{document}

\begin{abstract}
We study the inhomogeneous Landau equation with Coulomb potential and derive a new continuation criterion: a smooth solution can be uniquely continued for as long as it remains bounded. This provides, to our knowledge, the first continuation criterion based on a quantity not controlling the mass density. Consequently, we are able to rule out a potential singularity formation scenario known as \emph{tail fattening}, in which an implosion occurs due to the loss of decay at large $v$.

More generally, we are able to rule out all Type II approximately self-similar blow-up rates that are slower than the Type I blow-up rate, without any assumption of decay on the inner profile, complementing existing Type I blow-up analysis in the literature. Heuristically, this suggests that it should be impossible to directly use the hydrodynamic limit connection with the 3D compressible Euler equations to construct a singular solution to the Landau equation with Coulomb potential. Such a potential implosion scenario---based on either an isentropic or nonisentropic implosion for the 3D Euler equations---would naturally result in a slow Type II approximately self-similar blow-up scenario, falling well within the range our theorem.

This preprint has been subsumed by a more recent work by the authors and Luis Silvestre titled ``Pointwise bounds and obstructions to blowup for the Landau and Boltzmann equations,'' arXiv:2605.20426. This manuscript will remain a permanent preprint; all references should be directed to the more recent work.
\end{abstract}

\maketitle

\section{Introduction}

The Landau equation is an integro-differential equation and widely accepted fundamental model for the evolution of collisional plasmas. As a kinetic equation, the Landau equation describes the time evolution of an unknown density function $f(t,x,v)$ defined on phase space, which provides a statistical description of the media (in this case, an ionized gas, i.e. plasma), resulting in an equation of the form:
\begin{equation}\label{eq:landau}
    (\partial_t + v \cdot \nabla_x) f = Q_L(f) \qquad \text{for } (t,x,v) \in \R^+ \times \Omega \times \R^3,
\end{equation}
where $t \in \R^+$, is time, $x \in \Omega \subset \R^3$ is position, $v \in \R^3$ is velocity, and $Q_L(f)$ is the Landau collision operator, modeling microscopic particle interactions. Hereafter, the spatial domain $\Omega$ denotes either $\R^3$ or $\mathbb{T}^3$, the three dimensional torus. 
Additionally, the operator $Q_L$ is a second order, degenerate elliptic operator with nonlocal coefficients. It can be written several ways, depending on whether one wishes to emphasize collisional, divergence, or non-divergence structure. These are, respectively:
\be\label{eq:QL}
\begin{split}
    \QL(f)
        &
        := \frac{1}{8 \pi} \nabla_v \cdot\int \frac{\Pi(v-v_*)}{|v-v_*|} \left( f_* \nabla_v f - f \nabla_{v_*} f_* \right) \dd v_*
        \\[5pt]
        &
        = \nabla_v \cdot \Big( A[f] \nabla_v f - \nabla_v a[f] f\Big)
        \\[5pt]
        &
        = \Tr(A[f] D^2_v f) + f^2.
\end{split}
\ee
Here we use the standard notation of $f_*= f(v_*)$, $f = f(v)$, and
\be
    \Pi(w):= \Id - \frac{w\otimes w}{|w|^2}
\ee
for all $w\in\R^3$.  Additionally, we have implicitly introduced the notation
\be\label{eq:defn_coefficients}
    A[f] = \frac{\Pi}{8 \pi |\cdot|} * f
    \quad\text{ and }\quad
    a[f] = \frac{1}{4\pi |\cdot|} * f.
\ee
It is common to modify the collisional form of \eqref{eq:QL} or equivalently \eqref{eq:defn_coefficients}, introducing a family of models known as the Landau equation with a power law potential. The particular form in \eqref{eq:QL} and \eqref{eq:defn_coefficients} results in the Landau equation with Coulomb potential, which is by far the most physically relevant model. Interestingly, our results seem to hold only in the case of Coulomb potential.

One major long-standing problem in kinetic theory is determining whether \eqref{eq:landau} admits global-in-time smooth solutions for arbitrary initial data. We note that, at present, despite concerted effort, a conclusive answer to this question appears out of reach in either direction.

Instead, in the positive direction, most research has centered on constructing global very weak solutions (see \cite{Villani}), constructing global strong solutions in a perturbative regime (see \cite{Guo,CarrapatosoMischler,KimGuoHwang}), constructing local-in-time solutions (see \cite{HendersonSnelsonTarfulea,HendersonSnelsonTarfulea1,HeYang,Chaturvedi,SnelsonTaylor}), and identifying minimal constraints that prevent singularity formation, namely conditional regularity. This manuscript falls into the latter category and we summarize these results more thoroughly. Conditional regularity results typically make some form of the following assumption:
\begin{equation}\label{eq:Hydrodynamic_Assumption}\tag{Assumption H}
\begin{aligned}
    m_0 \le &\int_{\R^3} f(t,x,v) \dd v &\le M_0,\\
    &\int_{\R^3} \abs{v}^2 f(t,x,v) \dd v &\le E_0,\\
    &\int_{\R^3} f(t,x,v)\log f(t,x,v) \dd v &\le H_0,
\end{aligned}
\end{equation}
interpreted as pointwise bounds on the macroscopic or hydrodynamic quantities associated to $f$. Under \eqref{eq:Hydrodynamic_Assumption}, hypoelliptic Schauder estimates were employed to bound higher order regularity of $f$ in terms of a $C^{0,\alpha}$ norm (adapted to kinetic scaling) in \cite{HendersonSnelson}. Again under \eqref{eq:Hydrodynamic_Assumption}, De Giorgi type estimates yield control of the $C^{0,\alpha}$-norm by the $L^\infty$-norm (see \cite{GolseImbertMouhotVasseur}). Subsequent work by~\cite{HendersonSnelsonTarfulea1} showed that, in fact, the lower mass bound $m_0$ and the entropy upper bound $H_0$ in~\eqref{eq:Hydrodynamic_Assumption} can be dropped from \eqref{eq:Hydrodynamic_Assumption}. In summary, the problem of singularity formation is reduced to upper bounds on the mass density $M_0$, upper bounds on the energy density $E_0$, and a pointwise bound on $f$.

Further work has focused on weakening the extra assumption of a pointwise bound on $f$ as much as possible. In the so-called moderately soft potentials case---where the kernel in~\eqref{eq:defn_coefficients} is less singular---this extra assumption can be dropped entirely (see \cite{CameronSilvestreSnelson}). In the case of Coulomb potential considered here, the strongest continuation result at present can be found in \cite{SnelsonSolomon}, where the pointwise bound is replaced by an $L^p$ bound for any $p > \sfrac32$.

Heuristically, the aim of the program outlined above is to show that any potential singularity formation in the kinetic model must occur in the hydrodynamic quantities \eqref{eq:Hydrodynamic_Assumption} and, thus, should already be visible at the hydrodynamic limit; that is, in the Euler equations (see below). In the case of moderately soft potentials---where this approach is successful---no new form of blow-up is expected at the kinetic level. At present, it remains unclear whether any new, by which we mean fundamentally kinetic, blow-up should be expected for the Landau equation with Coulomb potential. Our results suggest that, surprisingly, the most likely blow-up mechanism for \eqref{eq:landau} is kinetic rather than hydrodynamic.

\subsection{Main results}

The first main contribution of the present manuscript is a continuation result in the spirit of the aforementioned works. 
\begin{theorem}\label{thm:main}
    Suppose that $f_{\rm in} \in \S(\Omega \times \R^3)$ take $f\in C^\infty([0,T_*);\S(\Omega \times \R^3))$ is the unique Schwartz class solution $f$ to the Landau equation~\eqref{eq:landau} with initial data $f_{\rm in}$ on its maximal interval $[0,T_*)$ of existence. Then,
    \begin{equation*}
            \int_0^{T_*} \norm{f(t)}_{L^\infty_{x,v}} \dd t = \infty.
    \end{equation*}
\end{theorem}

The main novelty of \Cref{thm:main} is that the continuation quantity, namely $\norm{f}_{L^1_tL^\infty_{x,v}}$, does not control the hydrodynamic quantities in \eqref{eq:Hydrodynamic_Assumption}. Correspondingly, \Cref{thm:main} rigorously rules out one potential mechanism for the formation of an implosion singularity: loss of decay in velocity, which we refer to as \emph{tail fattening}. We reiterate that our proof of \Cref{thm:main} at present applies only to the Coulomb case considered here.

For a local differential equation, loss of decay at infinity would not usually be termed singularity formation. However, we remind the reader that the Landau equation is non-local in $v$ and sufficient decay in $v$ is thus required to make sense of the collision operator $Q_L(f)$. Additionally, the physically relevant mass, momentum, and energy densities are defined via integrals in $v$. Consequently, a loss of decay in $v$ at a point $(t_0,x_0)$ is one mechanism by which the hydrodynamic quantities could become infinite at $(t_0,x_0)$. \Cref{thm:main} rigorously rules out this mechanism.

While tail fattening is a rather innocuous-looking form of blow-up, it is arises naturally in the hydrodynamic limit connecting the 3D compressible Euler equations to the Landau equation. In this limiting regime, solutions to \eqref{eq:landau} converge to local Maxwellians of the form:
\begin{equation*}
    \mu_{\rho,u,\theta} \coloneqq \frac{\rho(t,x)}{(2\pi\theta(t,x))^{3/2}}\exp\left(\frac{-\abs{v-u(t,x)}^2}{2\theta(t,x)}\right),
\end{equation*}
where the triple $(\rho,u,\theta)$ represents the mass, velocity, and temperature densities, respectively, and solve the full compressible 3D Euler equations for a monatomic ideal gas. Singularity formation for the 3D Euler equations is a substantially more developed and mature area than for kinetic equations. Indeed, while the only blow up result we are aware of for a collisional kinetic equation is \cite{Chen}, the first implosion singularities were constructed by Guderley in 1942 \cite{Guderley} and more rigorously constructed recently in \cite{JangLiuSchrecker,JangLiuSchrecker2}. The Guderley-type solutions rely on self-similar converging shock waves and are in some sense a stable mechanism for implosion (see \cite{CialdeaSchkollerVicol} for a more precise discussion). Additionally, smooth implosions were constructed in \cite{MerleRaphaelRodnianskiSzeftel,MerleRaphaelRodnianskiSzeftel2} and in a follow up \cite{BuckmasterCaoGomez}. Consequently, one natural approach to constructing singular solutions to Landau is to take a singular solution $(\rho, u, \theta)$ to the Euler equations and find a solution to Landau with leading order term $\mu_{\rho,u,\theta}$. We claim that \Cref{thm:main} implies this approach cannot succeed:
\begin{heuristic}\label{thm:heuristic}
    Suppose $(\rho,u,\theta)$ is any solution to the temperature dependent 3D Euler equations for a monatomic gas. Then, the ansatz
    \begin{equation*}
        f(t,x,v) = \mu_{\rho,u,\theta}(t,x,v)\qquad + \qquad \mathrm{lower\; order \;terms},
    \end{equation*}
    cannot produce a finite time blow-up solution $f(t,x,v)$ to the Landau equation with Coulomb potential.
\end{heuristic}

The above heuristic is related to a general phenomenon, called Type II approximately self-similar blow-up, where the blow-up rate is not dictated by scaling, but rather some other feature of the problem. Indeed, for the local Maxwellians above, the blow-up rate (computed in $L^\infty$), would be $\rho \theta^{-3/2}$, which is closely related to a physical quantity of a monatomic gas known as the specific entropy. As an additional consequence, \Cref{thm:main} yields a lower bound on the blow-up rate for self-similar scenarios:
\begin{corollary}\label{cor:selfsimilar}
    Suppose $f(t,x,v)$ is a Schwartz class solution to the Landau equation on $[-1,0) \times \Omega \times \R^3$ of the form
    \begin{equation*}
        f(t,x,v) = \frac{1}{\lambda(t)} g\left(\frac{x}{\nu(t)},\frac{v}{\mu(t)}\right) + h(t,x,v),
    \end{equation*}
    for three functions $\mu,\nu,\lambda$
    where $g \in L^\infty(\R^6)$ and $h\in L^1(-1,0;L^\infty_{x,v})$. If there holds
    \begin{equation}
        \limsup_{t\to 0^-} \frac{\lambda(t)}{\abs{t}} > \norm{g}_{L^\infty}
    \end{equation}
    then $f$ can be uniquely continued past $t = 0$ as a Schwartz class solution.
\end{corollary}
As a consequence, the only admissible blow-up rates are at least as fast as $\norm{g}_{L^\infty}/|t|$, which is the Type I blow-up rate.
 We remark that Type I blow-up was studied in \cite{BedrossianGualdaniSnelson}, however, the authors need to assume a decay condition on their inner profile $g$. By contrast, in \Cref{cor:selfsimilar}, no decay is assumed on our inner profile $g$. This is a rather desirable property because, often, inner profiles decay very slowly at infinity; see the discussion in \cite[Remark~2.3]{BedrossianGualdaniSnelson}.  
In order to produce a reasonable solution, however, lower order terms are typically included via $h$ and are used to cancel the slow decay of $g$ away from the blow-up point. We refer the reader to the proof of \Cref{cor:selfsimilar} below for a further discussion of scaling of \eqref{eq:landau}.

\begin{remark}
    Note that \Cref{cor:selfsimilar} leaves open the possibility of Type II blow-up rates that are strictly faster than the Type I rate, but satisfactorily rules out slower Type II rates. 
\end{remark}

We comment briefly on the proof of \Cref{thm:main}. The main new step is an estimate that allows us to propagate pointwise decay of the initial datum without any of the conditions in \ref{eq:Hydrodynamic_Assumption}:
\begin{proposition}\label{prop:upper_bound}
    Suppose $f\colon[0,T]\times \Omega\times \R^3 \to \R^+$ is a smooth, rapidly decaying solution to the Landau equation. Then, for each non-integer $m\in (2,5)$, $f$ satisfies the estimate
    \begin{equation}
        \norm{f(t)}_{L^\infty_m} \le \exp\left(K\int_0^t \norm{f(s)}_{L^\infty} \dd s\right)\norm{f_{\rm in}}_{L^\infty_m}, \qquad \text{for each }0 \le t\le T,
    \end{equation}
    where $K = K(m)$ is a constant depending only on $m$.
\end{proposition}

\Cref{prop:upper_bound} should be interpreted as an {\em a priori} estimate for smooth, rapidly decaying solutions. We do not attempt to find the broadest notion of solution for which our estimates hold. Note that the regularity and decay of $f$ enable us to integrate by parts in weighted estimates without any concern for whether integrals converge or not.

\subsection{Continuation in a broader class}

We now state a refinement of \Cref{thm:main} requiring less regularity and decay of the initial data. Before introducing it, we emphasize to the reader that the most important and novel aspect of this theorem is the continuation result, which does not require decay of $f$.
\begin{theorem}\label{thm:existence}
    Suppose $f_{\rm in} \in L^\infty_m$ for some $m > 5$ satisfies the ``mass core'' assumption: there exists $r>0$ and a point $(x_0,v_0)$ such that
    \be\label{eq:mass_core}
        f_{\rm in} \geq \1_{B_r(x_0,v_0)}.
    \ee
    Then, there is a time $T > 0$ depending only on $m$ and $\norm{f_{\rm in}}_{L^\infty_m}$ and function $f:[0,T]\times \Omega\times \R^3 \to \R^+$, solving~\eqref{eq:landau} classically such that the following hold.
    \begin{enumerate}
    
        \item {\bf (Regularity)} Given any compact set $\tilde \Omega \subset(0,T]\times \Omega$, there is $\alpha>0$, depending on $\tilde \Omega$, such that
        $f \in C^2_{\rm kin, loc} \cap C^{0,\alpha}_{\rm kin}(\tilde \Omega \times \R^3)$.

        \item {\bf (Matching with the initial data)} The density $f$ attains the initial data in the sense that $f(t) \to f_{\rm in}$ strongly in $L^1_{\rm loc}(\Omega \times \R^3)$.
        
        \item \label{i.continuation}
        {\bf (Continuation)}  
        If $T_*$ is the maximal time of existence of $f$, then either $T_*=\infty$ or
        \be
            \int_0^{T_*} \|f\|_{L^\infty_{x,v}} \dd t = \infty.
        \ee
    \end{enumerate}
\end{theorem}

The main new aspect of \Cref{thm:existence} is the continuation criterion \eqref{i.continuation}, which matches that of \Cref{thm:main} although allowing for a much broader class of initial data.

It is easy to see from the proof that one can construct {\em weak solutions} if we reduce the decay requirement to simply $m>2$.  See \cite[Theorem~1.3]{HST_Boltzmann_existence} for explicit details in the Boltzmann case. This is particularly interesting because these solutions exhibit some regularization -- they are locally H\"older continuous -- despite having potentially infinite mass and energy. However, it is not clear that the solutions are classical without stronger decay assumptions.  Indeed, the $L^\infty_m$-bound on $f$ is not strong enough to obtain matching upper and lower bounds on $A[f]$ when $m\leq 5$.  As a result, one cannot apply the transformation defined in \cite[Lemma~4.1]{CameronSilvestreSnelson} (see also \cite[Section~3]{HendersonSnelson} for the Coulomb case and improved regularity) to obtain sharp velocity decay of the H\"older norm of $f$, which is necessary to pass the regularity of $f$ to $A[f]$ and then apply the Schauder estimates.  Actually, it is only $(t,x)$-regularity that is at issue; $A[f]$ is defined by a convolution and so has some inherited $v$-regularity from the kernel.

In the homogeneous case, an easy adaptation of the proof of~\cite[Lemma~5.1]{GoldingLoher} by working in $L^2$ yields higher regularity of $f$; see also, the earlier work~\cite{GolseGualdaniImbertVasseur_PartialRegularity1}. This approach works because $x$-regularity is not at issue in the homogeneous case.  This readily gives the following result, whose proof we omit.

\begin{theorem}\label{thm:homogeneous}
    Suppose that $f_{\rm in}(x,v) = f_{\rm in}(v)$. 
    Under the assumptions of \Cref{thm:existence}, but allowing merely $m>2$, its conclusion holds and the solution is smooth.
\end{theorem}

In terms of decay, this is the weakest known condition for the construction of classical solutions (cf.~\cite{Ji2}, which requires $L^1_5$-boundedness of $f_{\rm in}$). Notably, our solution $f$ may have infinite mass and energy, depending on $m$.  On the other hand, in contrast to~\cite{Ji2}, our result requires pointwise boundedness and only provides a local solution when the Fisher information is initially infinite. 
It would be tempting to conjecture that the solution constructed in \Cref{thm:homogeneous} is global in time regardless of the finiteness of the Fisher information.

\subsection{Organization}

The remainder of the paper is organized as follows. In \Cref{sec:implosions}, we discuss the implications of \Cref{thm:main} in more detail. In particular, we discuss the scaling of \eqref{eq:landau}, deduce \Cref{cor:selfsimilar}, and combine \Cref{thm:main} with a few classical computations for hydrodynamic limits to explain \Cref{thm:heuristic}. In \Cref{sec:continuation}, we prove our main result \Cref{thm:main} from our main {\em a priori} bound, i.e. \Cref{prop:upper_bound}. In \Cref{sec:coefficients}, we then introduce a few bounds on the dissipation coefficient $A[f]$, defined in \eqref{eq:defn_coefficients}, in order to prove the estimate \Cref{prop:upper_bound}, which we do in \Cref{sec:main}. Finally, in \Cref{sec:existence}, we prove \Cref{thm:existence}.

\subsection{Notation}
Throughout the manuscript, we take the Japanese bracket convention $\brak{\cdot} \coloneqq (1 + \abs{\cdot}^2)^{1/2}$. 
Weighted $L^p$ spaces are defined using the bracket via the norm
\begin{equation*}
    \norm{f}_{L^p_m}^p \coloneqq \norm{\brak{\cdot}f}_{L^p}^p = \int_{\R^3} \brak{v}^{mp} f^p \dd v, 
\end{equation*}
with the standard modification when $p = \infty$. Whenever the domain of integration is omitted from an integral, the integral should be understood to be taken over the entire domain, i.e. $\Omega = \mathbb{T}^3$ or $\R^3$ if integrated with respect to $\dd x$ or $\R^3$ if integrated with respect to $\dd v$.  We use the notation $A \lesssim B$ if $A \leq C B$ for some universal constant $C$. If we wish to indicate further dependence on a parameter, say $\mu$, we write $A\lesssim_\mu B$.

\section{Consequences for Possible Implosion Singularities}\label{sec:implosions}

\subsection{Self-Similarity}

The Landau equation has a rather rich symmetry group. Beyond the rotation and translation symmetries associated with Galilean invariance, it exhibits a two parameter family of scaling symmetries that preserve the equation: If $f(t,x,v)$ solves \eqref{eq:landau}, then
\begin{equation*}
    f_{\mu,\lambda} \coloneqq \frac{1}{\lambda} f\left(\frac{t}{\lambda}, \frac{x}{\mu\lambda}, \frac{v}{\mu}\right) \qquad \text{solves \eqref{eq:landau} for any }\mu,\lambda > 0.
\end{equation*}
Taking $\lambda = t$, one sees that the natural blow-up rate determined by scaling is always $t^{-1}$. Consequently, even though~\eqref{eq:landau} admits a two parameter scaling symmetry group, there is only one possible Type I blow-up rate, namely $\lambda = t$, at least when the solution is measured in $L^\infty$. 
\begin{proof}[Proof of \Cref{cor:selfsimilar}]
For an approximately Type II blow-up scenario, the blow-up rate is potentially unrelated to scaling, resulting in the ansatz
\begin{equation}\label{eq:ansatz}
    f(t,x,v) = \frac{1}{\lambda(t)} g\left(\frac{x}{\nu(t)},\frac{v}{\mu(t)}\right) + h(t,x,v)
\end{equation}
for three functions $\mu,\nu,\lambda$ where $g \in L^\infty(\R^6)$ and $h\in L^1(-1,0;L^\infty_{x,v})$. Note that if $\lambda^{-1} \in L^1(-1,0)$, \Cref{thm:main} immediately applies. To address the case $\lambda^{-1}\in L^{1,\infty}(-1,0)$, we show that the maximum principle implies a minimal rate of blow-up.

Suppose that $f(t,x,v)$ is a smooth solution to \eqref{eq:landau} satisfying \eqref{eq:ansatz} and further assume for the sake of contradiction that $\limsup_{s\to 0^-} \norm{f(s)}_{L^\infty} = \infty$. Then, the maximum principle implies the $L^\infty$ norm satisfies the following Riccati-type differential inequality:
\begin{equation}\label{e.Riccati}
    \frac{\dd}{\dd t} \norm{f}_{L^\infty} \le \norm{f}_{L^\infty}^2.
\end{equation}
Let $y(t) = \norm{f(t)}_{L^\infty}$.  Using this notation, we rewrite~\eqref{e.Riccati} as
\begin{equation}
    - \frac{d}{dt} \frac{1}{y} \leq 1.
\end{equation}
Integrating this, we find, for any $-1 \le t \le s < 0$,
\begin{equation}\label{e.y}
    \frac{1}{y(t)}- \frac{1}{y(s)}
        \leq s-t.
\end{equation}
By assumption, there is a sequence of times $s_k$ such that $y(s_k) \to \infty$ as $s_k \to 0^-$. Evaluating~\eqref{e.y} with $s=s_k$ and taking the limit $k\to\infty$, we find
\begin{equation}
    \frac{1}{y(t)}
        \leq -t,
        \qquad\text{ or, equivalently, }\qquad
    \frac{1}{|t|} \leq y(t).
\end{equation}
Computing $y(t)$, we have
\begin{equation}
    \frac{1}{\abs{t}} \le \frac{\norm{g}_{L^\infty}}{\lambda(t)} + \norm{h(t)}_{L^\infty}.
\end{equation}
Since $h\in L^1(0,T;L^\infty)$, taking the limit as $t \to 0^-$, 
\begin{equation*}
    \limsup_{t\to 0^-} \frac{\lambda(t)}{\abs{t}} \le \norm{g}_{L^\infty},
\end{equation*}
a contradiction. It follows that $\displaystyle\limsup_{s\to 0^-} \norm{f(s)}_{L^\infty} < \infty$ and $f$ can be continued past time $0$ by Theorem \ref{thm:main}. The proof is finished.
\end{proof}

\subsection{The Compressible Euler equations and Specific Entropy}

The full (temperature-dependent) Euler equations are typically written in conservative form as
\begin{equation}\label{eq:Euler}\tag{3D Euler}
\begin{aligned}
    \partial_t \rho + \nabla_x \cdot (\rho u) &= 0\\
    \partial_t (\rho u) + \nabla_x \cdot (\rho u\otimes u +  pI) &= 0\\
    \partial_t (\rho E) + \nabla_x \cdot (\rho E u + pu) &= 0,
\end{aligned}
\end{equation}
where $\rho$ is the mass density, $u$ is the velocity field and $E$ is the total energy, i.e. $E = e + \frac{\abs{u}^2}{2}$, where $e$ is the internal energy. For a compressible fluid, one must also specify an equation of state for the pressure. The ideal gas law for a polytropic gas takes the form $p = (\gamma - 1)\rho e$, where $\gamma$ is the adiabatic constant of the gas. It is well-known that \eqref{eq:Euler} admits a physically-relevant entropy $S$:
\begin{equation*}
    S = \rho \log\left( \frac{\rho^{\gamma - 1}}{e}\right).
\end{equation*}
The entropy $S$ satisfies an inequality encoding the second law of thermodynamics, i.e. the decrease of entropy (with the mathematicians' sign convention):
\begin{equation*}
    \partial_t S + \nabla_x \cdot Q \le 0, \qquad Q  = \rho u\log \left(\frac{\rho^{\gamma - 1}}{e}\right).
\end{equation*}
The above equation holds with equality in smooth regions of fluid flow, but the inequality can become strict when discontinuities (i.e. shock waves) are present.
Arguing formally, the entropy inequality combined with the continuity equation yields: 
\begin{equation*}
\begin{aligned}
    \partial_t S + \nabla_x \cdot Q &=\rho\left(\partial_t \log\left( \frac{\rho^{\gamma - 1}}{e}\right) + u \cdot \nabla_x \log\left( \frac{\rho^{\gamma - 1}}{e}\right)\right) \le 0.
\end{aligned}
\end{equation*}
Consequently, the specific entropy $\overline{S} = \log\left( \frac{\rho^{\gamma - 1}}{e}\right)$ satisfies a pure transport equation:
\begin{equation*}
    \partial_t \overline{S} + u \cdot \nabla_x\overline{S} \le 0,
\end{equation*}
where equality holds in smooth regions of fluid flow. Note that if $u$ is smooth enough to admit a flow map, $\overline{S}$ is constant along flow lines and hence satisfies
\begin{equation*}
    \norm{\overline{S}(t)}_{L^\infty} \le \norm{\overline{S}(0)}_{L^\infty}, 
\end{equation*}
at least until the time of first blow-up. On the other hand, even in the presence of discontinuities or regions where $\overline{S} = -\infty$, one may still justify the above argument and conclude
\begin{equation*}
    \sup_{x \in \R^3}\overline{S}(t,x) \le \sup_{x\in \R^3} \overline{S}(0,x).
\end{equation*}
As a consequence,
\begin{equation*}
    \log\left( \frac{\rho^{\gamma - 1}}{e}\right) \lesssim 1, \qquad \text{or\;\;equivalently} \qquad \rho^{\gamma - 1} \lesssim e,
\end{equation*}
where the constant depends only on the initial bound on the specific entropy. Note that this inequality remains true even in the presence of vacuum, shocks, or regions of infinite temperature, where the specific entropy can become $-\infty$.

\subsection{The Strongly Collisional Hydrodynamic Limit}

The usual Boltzmann and Landau kinetic equations limit to the system \eqref{eq:Euler} under a particular regime called the strongly collisional limit. This process of deriving fluid equations from kinetic equations is known as the hydrodynamic limit and some of the earliest mathematical investigations were by Caflisch \cite{Caflisch}. The overview here follows the work of Bardos, Golse, and Levermore \cite{BardosGolseLevermore1,BardosGolseLevermore2}, in which the authors look at a kinetic equations of the form
\begin{equation*}
    (\partial_t + v \cdot \nabla_x) f = C(f),
\end{equation*}
where $C$ is a collision operator satisfying the very general assumptions:
\begin{itemize}
    \item Conservation of mass, momentum, and energy;
    \item Boltzmann's $H$-theorem, i.e. decrease of entropy;
    \item Maxwellian distributions are the only minimizers of the entropy, subject to the conservation law constraints.
\end{itemize}
Under these rather mild assumptions on the collision operator---which include the Landau and Boltzmann operators---one looks at the regime where collisions dominate: 
\begin{equation}\label{eq:kinetic}
     (\partial_t + v \cdot \nabla_x) f_{\eps} = \frac{1}{\eps}C(f_{\eps}), 
\end{equation}
and sends $\eps \to 0^+$. In the limit, one obtains a minimizer of the entropy, ensuring that the limiting distributions are local Maxwellians, i.e. functions of the form:
\begin{equation*}
    \mu_{\rho,u,\theta} \coloneqq \frac{\rho}{(2\pi \theta)^{3/2}} \exp\left(-\frac{\abs{v - u}^2}{2\theta}\right).
\end{equation*}
The form of $\mu$ ensures that $\mu$ depends on $t$ and $x$ only through the mass density $\rho$, the macroscopic velocity $u$, and the temperature density $\theta$. 
Additionally, the triple $(\rho,u,\theta)$ can be recovered as moments of $\mu_{\rho,u,\theta}$ via the the relation,
\begin{equation*}
    \begin{pmatrix}
        \rho \\ \rho u \\ \rho\left(\frac{\abs{u}^2}{2} + \frac{3}{2}\theta\right) 
    \end{pmatrix} = \int_{\R^3} \begin{pmatrix} 1 \\ v \\ \frac{1}{2}\abs{v}^2\end{pmatrix} \mu_{\rho,u,\theta}\; \dd v.
\end{equation*}
As a consequence of the microscopic equation \eqref{eq:kinetic}, conservation of mass, momentum, and energy, and computations of higher moments of Gaussians, the triple $(\rho,u,\theta)$ satisfies a closed system of equations:
\begin{equation}\label{eq:Euler2}
\begin{aligned}
    \partial_t \rho + \nabla_x \cdot (\rho u) &= 0\\
    \partial_t (\rho u) + \nabla_x \cdot (\rho u\otimes u +  (\rho \theta)I) &= 0\\
    \partial_t \left(\frac{\rho\abs{u}^2}{2} + \frac{3}{2}\rho\theta\right) + \nabla_x \cdot \left(\frac{\rho u\abs{u}^2}{2} + \frac{5\rho \theta u}{2}\right) &= 0.
\end{aligned}
\end{equation}
Defining the internal energy as $e \coloneqq \frac{3}{2}\theta$ and the total energy density as $E \coloneqq \frac{\abs{u}^2}{2} + e$, the system of equations \eqref{eq:Euler2} become the system \eqref{eq:Euler}, but with a specific pressure law:
\begin{equation*}
    p = \rho\theta = \frac{2}{3} \rho e = (\gamma - 1)\rho e \qquad \text{with }\gamma = \frac53.
\end{equation*}
\begin{NB}
    Under above assumptions on the collision operator $C(f)$, the kinetic equation \eqref{eq:kinetic} converges in the strongly collisional regime to the compressible Euler equations with adiabatic constant $\gamma = \sfrac53$. The Euler equations with other values of $\gamma$ are not observed as limits of the standard Landau and Boltzmann equations.
\end{NB}

\subsection{\Cref{thm:heuristic}}

Suppose now that one wishes to use the hydrodynamic limit described above to construct a singular solution to the Landau equation~\eqref{eq:landau}. A natural method to do so would be to take a solution $(\rho,u,e)$ to \eqref{eq:Euler} with the ideal gas law and make the ansatz
\begin{equation}\label{eq:ansatz2}
    f(t,x,v) = \mu_{\rho,u,\theta}(t,x,v) + \mathrm{lower\;order\; terms}.
\end{equation}
By the discussion of the hydrodynamic limit, one should not expect a meaningful result from this ansatz unless the pressure law for the Euler equations is taken to be $p = (\gamma - 1)\rho e$, with $\gamma = \sfrac53$ and $e = \frac{3}{2}\theta$. Therefore, analyzing the leading order term from the local Maxwellian,
\begin{equation*}
\begin{aligned}
    \mu_{\rho,u,\theta}(t,x,v) &= \frac{\rho}{(2\pi \theta)^{\sfrac32}} \exp\left(-\frac{\abs{v - u}^2}{2\theta}\right)\\
        &\sim \left(\frac{\rho^{\sfrac23}}{e}\right)^{\sfrac32}\exp\left(-\frac{\abs{v - u}^2}{2\theta}\right) \le \left(\frac{\rho^{\gamma - 1}}{e}\right)^{\sfrac32} \lesssim 1,
\end{aligned}
\end{equation*}
where the implicit constant depends on the initial specific entropy of the gas.
In other words, no matter the solution to the Euler equations---singular or not, isentropic or not---the leading order term in the ansatz \eqref{eq:ansatz2} is necessarily bounded uniformly in time. Consequently, attempting to produce self-similar blow-up for \eqref{eq:landau} using \eqref{eq:ansatz2} is a Type II blow-up mechanism ruled out by \Cref{cor:selfsimilar} or directly by \Cref{thm:main}.

\section{The Continuation Criterion: \Cref{thm:main}}\label{sec:continuation}

Before beginning the proof, we summarize a simplified version of the main result of \cite{HendersonSnelsonTarfulea}, which we use as a tool to establish \Cref{thm:main}.
\begin{theorem}\label{thm:chris_stan_andrei}
    For any Schwartz class initial datum $f_{\rm in} \in \S(\Omega \times \R^3)$, there is a $T>0$ and unique $f\colon[0,T)\times\Omega\times \R^3 \to \R^+$ a solving to \eqref{eq:landau} with initial datum $f_{\rm in}$ such that $f \in C^\infty(0,T;\S(\Omega\times \R^3))$. Moreover, $f$ may be uniquely continued past $T$ provided
    \begin{equation*}
        \sup_{0< t <T,\;x\in\Omega} \left(\sup_{v\in \R^3} f(t,x,v) +\int_{\R^3} f(t,x,v) \dd v\right) < \infty.
    \end{equation*}
\end{theorem}

\begin{proof}[Proof of \Cref{thm:main}]
Suppose that 
\begin{equation*}
    \int_0^T \norm{f(t)}_{L^\infty_{x,v}} \dd t < \infty.
\end{equation*}
Then, by the {\em a priori} estimate in \Cref{prop:upper_bound}, for any $m \in (2,5)$ non-integer, we find
\begin{equation*}
     f(t,x,v) \le \frac{C(m,T,f_{\rm in})}{\brak{v}^m}, \qquad \text{for any }(t,x,v) \in [0,T)\times \Omega\times \R^3.
\end{equation*}
The proof is clearly finished after applying \Cref{thm:chris_stan_andrei}.
\end{proof}

\section{Weighted Coefficient Bounds}\label{sec:coefficients}

To derive our main {\em a priori} estimate, we need precise weighted estimates on the diffusion coefficient of the Landau equation. We recall the following convolution estimate from \cite{GoldingHenderson_LFD}:

\begin{lemma}
    [{\cite[Lemma 2.4]{GoldingHenderson_LFD}}]
    \label{l.polynomial_convolution}
Fix $\alpha \in (0,3)$. Then, for any $v \in \R^3$, we have the following bound:
\begin{equation}
    \int_{\R^3} \frac{1}{\abs{v - w}^{\alpha}\brak{w}^m} \dd w
        \lesssim_{\alpha,m}
        \begin{cases}
            \brak{v}^{3 - \alpha - m}
                \quad &\text{ if }3-\alpha < m < 3\\
            \vv^{-\alpha} \log(1 + \vv) 
                \quad &\text{ if } m = 3,\\
            \brak{v}^{-\alpha}
                \quad &\text{ if }m > 3.
    \end{cases}
\end{equation}
\end{lemma}

With \Cref{l.polynomial_convolution} in hand, we are prepared to begin bounding the nonlocal coefficient $A[h]$.
Our first estimate is a pointwise in velocity estimate on the diffusion coefficient:
\begin{lemma}\label{l.A_Linfty_m}
    Fix a non-integer $m \in (2,5)$ and $h \in L^\infty_m$ with $h \ge 0$. Then, for any $v\in\R^3$,
    \begin{equation}
        v \cdot A[h]v \lesssim \brak{v}^{4-m}\norm{h}_{L^\infty_m},
    \end{equation}
    where the implicit constant depends only on $m$.
\end{lemma}

The assumption that $m$ is a non-integer allows us to avoid log factors as in \Cref{l.polynomial_convolution}.

\begin{proof}
    Fix $m \in (4,5)$.  Using that $\Pi(z)z = 0$, $\Pi(z)$ is symmetric, and \Cref{l.polynomial_convolution} with $\alpha = 1$ yields
    \begin{equation}
    \begin{aligned}
        v\cdot A[h]v
            &= \int_{\R^3} \frac{\left[v \cdot \Pi(v - w)v\right]h(w)}{\abs{v - w}} \dd w
            = \int_{\R^3} \frac{\left[w\cdot \Pi(v - w)w\right]h(w)}{\abs{v - w}} \dd w
            \\&
            \le \int_{\R^3} \frac{\abs{w}^2\|h\|_{L^\infty_m}}{\abs{v - w}\brak{w}^m} \dd w
            \le \|h\|_{L^\infty_m}\int_{\R^3} \frac{1}{\abs{v - w}\brak{w}^{m-2}} \dd w
            \lesssim_m \brak{v}^{4-m}\|h\|_{L^\infty_m}.
    \end{aligned}
    \end{equation}
    The proofs for $m \in (2,3)$ or $(3,4)$ are derived 
    in the same manner, but exchanging $v$ for $w$ fewer times.
\end{proof}

We continue with a second pointwise in velocity estimate on the diffusion coefficient:
\begin{lemma}\label{l.A_Linfty_Lp}
    Fix any $m > 2$ and any $p > \sfrac{3}{2}$. If $h \in L^p_m$, then
    \begin{equation}
        \norm{A[h]}_{L^\infty}
            \lesssim_{m,p} \norm{h}_{L^p_m}.
    \end{equation}
\end{lemma}

\begin{proof}
    Applying \Cref{l.polynomial_convolution} with $\alpha = \frac{p}{p-1}$ and $m^* = \frac{pm}{p-1}$, we find
    \begin{equation}
    \begin{aligned}
        \abs{A[h]}
        \le 
        \int_{\R^3} \frac{\brak{w}^{m}h(w)}{\abs{v - w}\brak{w}^m} \dd w &\lesssim \norm{h}_{L^p_m}\left(\int_{\R^3} \frac{1}{\abs{v - w}^{\frac{p}{p-1}}\brak{w}^{m^*}} \dd w\right)^{\frac{p-1}{p}} \lesssim \norm{h}_{L^p_m}
    \end{aligned}
    \end{equation}
    provided $\alpha < 3$ and $3 - \alpha <  m^*$. The first constraint follow because $p>\sfrac32$, while the second constraint amounts to the condition $3 < \frac{p(m+1)}{p-1}$, which holds because $m > 2$. 
\end{proof}

\begin{remark}
    The above lemma is wildly suboptimal. However, this estimate is only applied to terms that are then controlled by qualitative regularity of a solution to the Landau equation. Using a sharper version of the above estimate only enlarges the class of weak solutions for which our {\em a priori} estimates hold.
\end{remark}

\section{Main Computation}\label{sec:main}

In this section, we prove our main {\em a priori} upper bound,  \Cref{prop:upper_bound}. 
The proof is based crucially on the following computation concerning the evolution of weighted $L^2$-norms:
\begin{lemma}\label{lem:stampacchia_estimate}
    Suppose $f\colon[0,T]\times \Omega\times \R^3 \to \R^+$ is a smooth, rapidly decaying solution to the Landau equation. For any non-integer $2 < m < 5$ and any $\ell:[0,T]\to \R$ Lipschitz, define the quantities
    \begin{equation}
        g = \brak{v}^m f \qquad \text{and} \qquad g_\ell = (g - \ell)_+ = \max(g-\ell,0).
    \end{equation}
    Then, $g$ satisfies the following differential inequality:
    \begin{equation}
    \begin{aligned}
        \int g_\ell \partial_t g \dd v \dd x
            &\lesssim_m \ell(t) \norm{f(t)}_{L^\infty_{x,v}} \int g_\ell \dd v \dd x + \left(\norm{f(t)}_{L^\infty_{x,v}} + \norm{\brak{\cdot}^mf(t)}_{L^\infty_x L^2_v}\right)\int g_\ell^2 \dd v \dd x.
    \end{aligned}
    \end{equation}
\end{lemma}

We first deduce \Cref{prop:upper_bound} from \Cref{lem:stampacchia_estimate}. Then, we spend the remainder of the section deriving the differential inequality contained in \Cref{lem:stampacchia_estimate}.

\subsection{Proof of \Cref{prop:upper_bound}}

Fix any non-integer $m \in (2,5)$. Choose $\ell(t)$ as 
\begin{equation}
    \ell(t) = \exp\left(K \int_0^t \norm{f(s)}_{L^\infty} \dd s\right)\norm{f_{\rm in}}_{L^\infty_m},
\end{equation} where $K$ is the corresponding implicit constant from \Cref{lem:stampacchia_estimate}. Consequently, $\ell$ satisfies the differential equation
\begin{equation*}
    \ell^\prime(t) = K\norm{f(t)}_{L^\infty}\ell(t) \qquad \text{and} \qquad \ell(0) = \norm{f_{\rm in}}_{L^\infty_m}.
\end{equation*}
Applying \Cref{lem:stampacchia_estimate} to $f$ and using the choice of $\ell$, we find
\begin{equation*}
\begin{split}
    \frac{1}{2} \frac{\dd}{\dd t} \int g_\ell^2 \dd v \dd x &= \int g_\ell \partial_t g \dd v \dd x - \ell^{\prime}(t)\int g_\ell \dd v \dd x\\
        &
            \leq K\left[\norm{f(t)}_{L^\infty_{x,v}} + \norm{\brak{\cdot}^mf(t)}_{L^\infty_x L^2_v}\right] \int g_\ell^2 \dd v \dd x.
\end{split}
\end{equation*}
Setting $y(t) = \int g_\ell^2 \dd v \dd x$, $y(t)$ satisfies the differential inequality
\begin{equation*}
    y^\prime(t) \le 2K\left[\norm{f(t)}_{L^\infty_{x,v}} + \norm{\brak{\cdot}^mf(t)}_{L^\infty_x L^2_v}\right]\;y(t), \qquad \text{and} \qquad y(0) = 0.
\end{equation*}
Because $f$ is qualitatively smooth and rapidly decaying, we may apply Gr\"onwall's inequality to conclude that $y(t) = 0$ for each $t \in [0,T]$. Consequently, $g_\ell(t,x,v)$ is identically $0$. After unraveling definitions, this implies the desired pointwise bound
\begin{equation*}
    f(t,x,v) \le \frac{\ell(t)}{\brak{v}^m} = \frac{\norm{f_{\rm in}}_{L^\infty_m}}{\brak{v}^m} \exp\left(K\int_0^t \norm{f(s)}_{L^\infty} \dd s \right),
\end{equation*}
which concludes the proof.

\subsection{Bound on the Principal Term}

In deriving \Cref{lem:stampacchia_estimate}, we encounter one problematic term frequently. We bound this term now using the weighted coefficient bounds from \Cref{l.A_Linfty_m} and \Cref{l.A_Linfty_Lp}: 
\begin{lemma}\label{l.g_ell_f_vAv}
    Under the assumptions of \Cref{lem:stampacchia_estimate}, there holds
    \begin{equation}
    \begin{aligned}
        \int g_\ell f\brak{v}^m \left(\frac{v}{\vv^2}\cdot A[f]\frac{v}{\vv^2}\right) \dd v \dd x
        \lesssim_m
            \ell \int f g_\ell \dd v \dd x
            + \norm{\brak{\cdot}^{m-2}f}_{L^\infty_x L^2_v} \int g_\ell^2 \dd v \dd x.
    \end{aligned}
    \end{equation}
\end{lemma}

\begin{proof}
    First, using $g = \brak{v}^m f$ and $g_\ell = (g - \ell)_+$, we decompose $f$ as 
    \begin{equation}\label{eq:g_ell_ellv}
        \begin{split}
            f
            &= \vv^{-m} g_\ell + \min(f,\ell \brak{v}^{-m})
            \leq \vv^{-m} g_\ell + \ell \brak{v}^{-m}.
        \end{split}
    \end{equation}
    Combining \eqref{eq:g_ell_ellv} with \Cref{l.A_Linfty_m} and \Cref{l.A_Linfty_Lp} with $p = 2$, we obtain the pointwise bound
    \begin{equation}
    \begin{aligned}
        v\cdot A[f]v 
        &
        \le v\cdot A[\brak{v}^{-m}g_\ell]v
        + \ell \left(v\cdot A[\brak{v}^{-m}]v\right)
        \lesssim \brak{v}^2\norm{g_\ell}_{L^2_v} + \ell \brak{v}^{4-m}.
    \end{aligned}
    \end{equation}
    Consequently, integrating in both $v$ and $x$, we obtain
    \begin{equation}
    \begin{aligned}
        \iint g_\ell f\brak{v}^{m-4}& \left(v\cdot A[h]v \right) \dd v \dd x
        \lesssim \int\norm{g_\ell}_{L^2_v}\int  \brak{v}^{m-2} g_\ell f \dd v \dd x
        + \ell \int f g_\ell \dd v \dd x.
    \end{aligned}
    \end{equation}
    To conclude, we bound the first integral using H\"older's inequality in $v$ and $x$:
    \begin{equation}
    \begin{split}
        \int \norm{g_\ell}_{L^2_v}\int \brak{v}^{m-2} g_\ell f \dd v \dd x
        &\lesssim \int\norm{f \vvd^{m-2}}_{L^2_v}\norm{g_\ell}_{L^2_{v}}^2 \dd x \le \norm{f \vvd^{m-2}}_{L^\infty_xL^2_v}\norm{g_\ell}_{L^2_{x,v}}^2.
    \end{split}
    \end{equation}
\end{proof}

\subsection{Proof of \Cref{lem:stampacchia_estimate}}

We now derive the key differential inequality in \Cref{lem:stampacchia_estimate} using $L^2$ energy estimates. Throughout this section, we fix $f:[0,T]\times \Omega \times \R^3 \to \R^+$ a smooth, rapidly decaying solution to the Landau equation, $m \in (2,5)$, and $\ell:[0,T]\to \R^+$ Lipschitz.
Associated to these choices, we define the following quantities appearing in \Cref{lem:stampacchia_estimate}:
\begin{equation}
    g = \vv^m f
    \quad\text{ and }\quad
    g_\ell = (g - \ell)_+.
\end{equation}
Our first step is to compute the equation satisfied by $g$:
\begin{equation}\label{eq:g_Landau}
\begin{aligned}
    (\partial_t + v\cdot \nabla_x) g
        &= \brak{v}^m \nabla_v \cdot \left(A[f]\nabla_v \frac{g}{\brak{v}^m} - \frac{g}{\brak{v}^m}\nabla a[f]\right)\\
        &=
        \brak{v}^m\nabla_v \cdot \left(\frac{1}{\brak{v}^{m}}A[f]\nabla_v g - \frac{m}{\brak{v}^{m}} gA[f]\frac{v}{\brak{v}^{2}}\right)
        - \nabla_v \cdot \left(g\nabla a[f]\right) + mg \left(\frac{v}{\brak{v}^2} \cdot \nabla a[f]\right)\\
        &= \nabla_v \cdot \left(A[f]\nabla_v g - g\nabla a[f]\right)  - m g\brak{v}^{-2}\Tr\left(A[f]\right)\\
        &\qquad - 2m\left(\frac{v}{\brak{v}^{2}} \cdot A[f]\nabla_v g\right) + m(m+2)g\left(\frac{v}{\brak{v}^2}\cdot A[f]\frac{v}{\brak{v}^2}\right).
\end{aligned}
\end{equation}
To perform an $L^2_{x,v}$ energy estimate, we multiply the above equation by $g_{\ell}$ and integrate in both $x$ and $v$. The free transport is a pure divergence and vanishes.  Additionally, notice that $\nabla_v g = \nabla_v g_\ell$ on the support of $g_\ell$. Hence, after rearranging terms, we find:
\begin{equation}\label{eq:RHS_A_RHS_a}
\begin{aligned}
    \int g_\ell\partial_t g \dd v \dd x
    &= - \int \nabla_v g_\ell \cdot A[f]\nabla_v g_\ell \dd v \dd x
    -2m \int g_\ell \nabla g_\ell \cdot A[f]\frac{v}{\brak{v}^2} \dd v \dd x
    \\
    &\qquad - m\int \vv^{-2} g_\ell g \Tr(A[f]) \dd v \dd x
    + m(m+2)\int g_\ell g\left(\frac{v}{\brak{v}^2}\cdot A[f]\frac{v}{\brak{v}^2}\right) \dd v \dd x
    \\
    &\qquad
    + \int  g \nabla_v g_\ell \cdot \nabla a[f]\dd v\dd x
    := \sum_{j=1}^5 I_j.
\end{aligned}
\end{equation}
Note that $I_1$ and $I_3$ have a good sign since $A[f]$ is a positive definite, symmetric matrix. For $I_2$, we use the Cauchy-Schwarz inequality in the form, for any $\delta > 0$,
\begin{equation*}
    I_2 \le  m \delta \int \nabla g_\ell \cdot A[f]\nabla g_\ell \dd v \dd x + \frac{m}{\delta}\int g_\ell^2 \left(\frac{v}{\brak{v}^2} \cdot A[f]\frac{v}{\brak{v}^2} \right) \dd v \dd x.
\end{equation*}
Picking $\delta = \frac{1}{m}$ and using the inequality $g_\ell \le g$, the first four terms are bounded as
\begin{equation*}
    \sum_{j=1}^4 I_j \le \left[m(m+2) + m^2 \right] \int g_\ell g \left(\frac{v}{\brak{v}^2} \cdot A[f]\frac{v}{\brak{v}^2} \right) \dd v \dd x.
\end{equation*}
For the final term involving $\nabla a$, we use the identity $g = g_\ell + \min(g,\ell)$, integrate by parts, and use $\Delta a[f] = -f$:
\begin{equation*}
\begin{aligned}
    I_5 &= \int g_\ell \nabla_v g_\ell \cdot \nabla a[f] \dd v \dd x + \ell\int \nabla_v g_\ell \cdot \nabla a[f] \dd v \dd x\\
        &= \frac{1}{2}\int g_\ell^2 f \dd v \dd x + \ell\int g_\ell f\dd v \dd x.
\end{aligned}
\end{equation*}
Combined with \Cref{l.g_ell_f_vAv}, we conclude that
\begin{equation*}
    \int g_\ell \partial_t g \lesssim \left(\norm{\brak{\cdot}^mf(t)}_{L^\infty_xL^2_v} + \norm{f(t)}_{L^\infty_{x,v}}\right)\int g_\ell^2 \dd v \dd x + \ell\norm{f(t)}_{L^\infty_{x,v}}\int g_\ell \dd v \dd x,
\end{equation*}
as desired.

\section{Continuation in a larger class: \Cref{thm:existence}}\label{sec:existence}

\begin{proof}[Proof of \Cref{thm:existence}]
The proof of this theorem is similar to the results in, e.g., \cite{HendersonSnelsonTarfulea, HST_Boltzmann_existence, SnelsonTaylor}.  The overview of the proof is to (1) regularize the initial data in order to apply previous existence results; (2) obtain a ``nice'' bound, uniform in the approximation, that holds on an $O(1)$ length time interval; (3) apply the continuation criteria with the bound from step (2) to extend the solution for the entirety of the $O(1)$ length time interval; (4) take the limit to remove the regularization parameter with the available smoothing results for the Landau equation.  As such, we merely sketch the proof.

\begin{flushleft}
    \textbf{Step 1: Uniform Bounds for an Approximate Solution}
\end{flushleft}

Add a small parameter $\eps\in(0,1)$ quantifying a regularization and localization of the initial data to obtain a family $f_{{\rm in},\eps}$ of non-negative Schwartz class initial data verifying
\begin{equation}\label{eq:Linfty_m_mollified}
    \|f_{{\rm in},\eps}\|_{L^\infty_m}
        \leq 2\|f_{\rm in}\|_{L^\infty_m}, \qquad \text{and} \qquad f_{{\rm in},\eps} \to f_{\rm in} \text{ strongly in $L^1_{\rm loc}$}.
\end{equation}
Applying \Cref{thm:chris_stan_andrei}, there is a $T_{\eps}$ and a unique Schwartz class solution $f_{\eps}\colon [0,T_{\eps}]\times \Omega \times \R^3 \to \R^+$ to the Landau equation with initial datum $f_{{\rm in},\eps}$.
Now, the {\em a priori} bound in  \Cref{prop:upper_bound} can be applied to $f_{\eps}$ with ~\eqref{eq:Linfty_m_mollified} to conclude: for every $t \in [0,T_{\eps}]$,
\begin{equation}
\norm{f_{\eps}(t)}_{L^\infty_m}
    \le 2\norm{f_{\rm in}}_{L^\infty_m} \exp\Big(K\int_0^t 
        \norm{f_{\eps}(s)}_{L^\infty} \dd s\Big).
\end{equation}
Setting $T_0 = \frac{\log(2)}{2K\norm{f_{\rm in}}_{L^\infty_m}}$, it is easy to check that
\begin{equation}\label{eq:bound}
\|f_{\eps}(t)\|_{L^\infty_m}
        \le 4\|f_{\rm in}\|_{L^\infty_m}, \qquad \text{for\;\;}t \in [0,\min\{T_0,T_\eps\}].
\end{equation}
Applying \Cref{thm:main}, the maximal time interval of existence can be extended beyond $T_\eps$.  In fact, we see that a solution exists at least on $[0,T_0]$.  Note that $T_0$ does not depend on $\eps$.

\begin{flushleft}
    \textbf{Step 2: Identifying the Limit and Regularity Estimates}
\end{flushleft}

We begin by identifying $f:[0,T_0]\times \Omega \times \R^3 \to \R$, the weak star limit of (a subsequence of) the family $f_{\eps}$ as $\eps \to 0$ in $L^\infty(0,T_0; L^\infty_m)$. 
Consequently, $f$ is non-negative and verifies the {\em a priori} estimate \eqref{eq:bound} with $\eps = 0$.

Arguing as in \cite[Proof of Theorem~1.2]{HendersonSnelsonTarfulea}, we obtain local H\"older regularity and $C^2_{\rm kin, loc}$ bounds.  Roughly, the first step of the argument is to use mass-spreading to obtain uniform-in-$\eps$ ellipticity for $A[f_{\eps}]$ from the uniform-in-$\eps$ upper bounds on $A[f_{\eps}]$ provided by \eqref{eq:bound}. 
Second, one uses the De Giorgi result of~\cite{GolseImbertMouhotVasseur} and a clever change of variables due to~\cite{CameronSilvestreSnelson} to obtain uniform-in-$v$ H\"older regularity of $f_{\eps}$. Third, it follows that $A[f_{\eps}]$ is H\"older regular, and the Schauder estimates of, e.g.,~\cite{HendersonSnelson, ImbertMouhot} can be applied: for any compact $\tilde \Omega \subset (0,T]\times \Omega\times \R^3$, one obtains $f_{\eps} \in C_{\rm kin}^{2,\alpha}(\tilde \Omega)$, uniformly in $\eps$, but where the value of $\alpha$ may depend on $\tilde \Omega$.  Regardless, this is sufficient to take the limit $\eps \to 0$ and obtain a classical solution $f \in C^2_{\rm kin, loc}$ of~\eqref{eq:landau}.

The skeptical reader might note that \cite[Theorem~1.2]{HendersonSnelsonTarfulea} is stated only for $m > \sfrac{15}2$ and with a stronger mass core lower bound.  These are bottlenecks due to the application, in the proof of \cite[Theorem~1.2]{HendersonSnelsonTarfulea}, of an earlier continuation result that we do not require here.  An inspection of the proof in~\cite{HendersonSnelsonTarfulea} reveals $m>5$ is sufficient for the parts of the proof recalled in the previous paragraph.

\begin{flushleft}
    \textbf{Step 3: Behavior at Initial Time}
\end{flushleft}

To verify that $f$ obtains the initial datum strongly, we require estimates that do not degenerate at $t = 0$. Fix $\widetilde\Omega \subset \Omega \times \R^3$ a compact set. Then, certainly $f_{\eps} \in L^\infty(0,T;L^2_{x,v}(\widetilde\Omega))$ uniformly in $\eps$ and by a standard duality argument, using the uniform-in-$\eps$ bounds on $\norm{f_{\eps}}_{L^\infty}$, $\norm{A[f_\eps]}_{L^\infty}$, and $\norm{\nabla a[f_\eps]}_{L^\infty}$ provided by \eqref{eq:bound},
\begin{equation*}
    \partial_t f_{\eps}  = -v\cdot \nabla_x f_{\eps} + \nabla_v^2 :\left(A[f_{\eps}]f_{\eps}\right) + 2\nabla_v \cdot \left(\nabla_v a[f_{\eps}]f_{\eps}\right) \in L^\infty(0,T; W^{-2,\infty}(\widetilde\Omega)).
\end{equation*}
The Aubin-Lions Lemma then guarantees $f_{\eps} \in L^\infty(0,T;H^{-1}(\widetilde\Omega))$ uniformly in $\eps$ and $f$ is continuous in the sense of distributions and consequently weakly continuous in $L^2_{x,v}(\widetilde\Omega)$ with the correct initial datum. 
On the other hand, an $L^2$ energy estimate using a cutoff $\varphi \in C^\infty_c(\widetilde \Omega)$ yields
\begin{equation*}
\begin{aligned}
    \int \varphi^2 f_\eps^2(t) \dd v \dd x
        + &2\int_0^t\int \varphi^2\nabla f_\eps \cdot A[f_\eps]\nabla f_\eps \dd v \dd x \dd s
        \\&
        =
        \int \varphi^2 f_{{\rm in},\eps} \dd v \dd x
        - 4\int_0^t\int \varphi f_\eps\nabla \varphi \cdot A[f_\eps]\nabla f_\eps \dd v \dd x \dd s
        \\&\qquad
        +\int_0^t\int \varphi^2f_{\eps}^3 \dd v\dd x \dd s
        + 4\int_0^t\int\varphi f^2_\eps\nabla \varphi \cdot  \nabla a[f] \dd v \dd x \dd s.
\end{aligned}
\end{equation*}
Applying Cauchy-Schwarz and using the uniform-in-$\eps$ bounds on $\norm{f_{\eps}}_{L^\infty},\;\norm{A[f_{\eps}]}_{L^\infty},\text{ and }\norm{\nabla a[f_{\eps}]}_{L^\infty}$, we obtain
\begin{equation*}
    \int \varphi^2 f_\eps^2(t) \dd v \dd x + \int_0^t\int \varphi^2\nabla f_\eps \cdot A[f_\eps]\nabla f_\eps \dd v \dd x \dd t \le \int \varphi^2 f_{\rm in}^\eps \dd v \dd x + C\,t.
\end{equation*}
We send $\eps \to 0^+$ first using the strong convergence \eqref{eq:Linfty_m_mollified} and then sending $t\to 0^+$ to conclude
\begin{equation*}
    \limsup_{t \to 0^+} \int \varphi^2 f^2(t) \dd v \dd x
        \le \int \varphi^2 f^2_{\rm in}\dd v \dd x.
\end{equation*}
We conclude that $f^2(t) \to f_{\mathrm{\rm in}}^2$ in the sense of distributions, which implies $f^2(t) \to f_{\rm in}^2$ weakly in $L^1(\widetilde\Omega)$, since $\set{f^2(t)}_{t>0}$ is uniformly integrable by the $L^\infty$-bound. Combined with $f(t) \to f_{\rm in}$ weakly in $L^2(\widetilde\Omega)$, $f(t) \to f_{\rm in}$ strongly in $L^2(\widetilde\Omega)$.

\begin{flushleft}
    \textbf{Step 4: Continuation}
\end{flushleft}

This follows by showing that, if $f$ solves~\eqref{eq:landau} on $[0,\tilde T]\times \Omega \times \R^3$, for some $\tilde T>0$, and 
\be\label{e.c110101}
    \int_0^{\tilde T} \|f(t)\|_{L^\infty_{x,v}} \dd t < \infty,
\ee
then $f$ can be extended to $[0, \tilde T+\delta]$, where $\delta>0$ depends only on the left hand side of~\eqref{e.c110101} and $\|f_{\rm in}\|_{L^\infty_m}$.  Indeed, from~\eqref{e.c110101} and \Cref{prop:upper_bound}, we see that $f(\tilde T) \in L^\infty_m$.  Moreover, the mass spreading result discussed in Step 2, ensures that the mass core assumption~\eqref{eq:mass_core} is satisfied at time $\tilde T$.  Applying the already-proved existence portion of the theorem with $f(\tilde T)$ as the initial data yields a solution on $[0,\tilde T+\delta]$.

One needs only verify that $f$ remains a classical solution across the time $t=\tilde T$.  It is straightforward to check that $f$  is a weak solution on $[0,\tilde T + \delta]\times \Omega\times \R^3$.  Hence, the De Giorgi estimates of~\cite{GolseImbertMouhotVasseur} and Schauder theory of~\cite{HendersonSnelson} apply and yield the $C^2_{\rm kin, loc}$ regularity of $f$ on $(0,\tilde T + \delta]\times \Omega \times \R^3$.  This concludes the proof.
\end{proof}

\bibliographystyle{abbrv}
\bibliography{LandauRefs}

\end{document}